\documentclass[12pt]{article}
\usepackage{amsfonts,latexsym,rawfonts,amsmath,amssymb,amsthm}
\usepackage{amsmath,amssymb,amsfonts,latexsym,lscape,rawfonts}
\usepackage[cm]{fullpage}
\usepackage{amsmath,amscd, float,times,rotating}
\usepackage{pb-diagram}
\usepackage{color}
\usepackage{hyperref}
\numberwithin{equation}{section}

\newcommand{\pd}[2]{\frac {\partial #1}{\partial #2}}

\newcommand{\La}{\Lambda}
\newcommand{\oo}{\omega}

\newcommand{\si}{\sigma}

\newcommand{\beq}{\begin{equation}}
\newcommand{\eeq}{\end{equation}}
\newcommand{\beqs}{\begin{eqnarray*}}
\newcommand{\eeqs}{\end{eqnarray*}}
\newcommand{\beqn}{\begin{eqnarray}}
\newcommand{\eeqn}{\end{eqnarray}}
\newcommand{\beqa}{\begin{array}}
\newcommand{\eeqa}{\end{array}}

\def\td{\tilde}
\def\p{\partial}

\def\RR{{\mathbb R}}

\def\CC{{\mathbb C}}

\def\ri{\rightarrow}

\def\si{\sigma}
\def\pbp{\sqrt{-1}\partial\bar\partial}

\def\cF{{\mathcal F}}

\def\cM{{\mathcal M}}

\def\i{{\sqrt{-1}}}
\def\n{\nonumber}
\newtheorem{prop}{Proposition}[section]
\newtheorem{theo}[prop]{Theorem}
\newtheorem{lem}[prop]{Lemma}
\newtheorem{claim}[prop]{Claim}

\newtheorem{rem}[prop]{Remark}

\newtheorem{defi}[prop]{Definition}

\newtheorem{q}[prop]{Question}

\title{Energy functionals and K\"ahler-Ricci solitons}
\author{Haozhao Li}

\begin{document}

\maketitle
\date{}

\bigskip

 \noindent {\bf Abstract} \ In this paper, we generalize Chen-Tian energy functionals to
K\"ahler-Ricci solitons and prove that the properness of these
functionals is equivalent to the existence of K\" ahler-Ricci
solitons. We also discuss the equivalence of the lower boundedness
of these functionals and their relation with Tian-Zhu's holomorphic
invariant.

\section{Introduction}
In \cite{[chen-tian1]}, a series of energy functionals $E_k(k=0, 1,
\cdots, n)$ were introduced by X.X. Chen and  G. Tian which were
used to prove the convergence of the K\"ahler Ricci flow under some
curvature assumptions. The first energy functional $E_0$ of this
series is exactly the $K$-energy introduced by Mabuchi in
\cite{[Ma]}, which can be defined for any K\"ahler potential
$\varphi(t)$ on a K\"ahler manifold $(M, \oo)$ as follows:
$$\frac {d}{dt}E_0(\varphi(t))=-\frac 1V\int_M\;\pd {\varphi}t(R_{\varphi}-r)\oo_{\varphi}^n.$$
Here $R_{\varphi}$ is the scalar curvature with respect to the
K\"ahler metric $\oo_{\varphi}=\oo+\pbp \varphi$,   $r=\frac
{[c_1(M)]\dot[\oo]^{n-1}}{[\oo]^n}$ is the average of $R_{\varphi}$
and $V=[\oo]^n$ is the volume.

It is well-known that the behavior of the $K$-energy plays a central
role on the existence of K\"ahler-Einstein metrics and constant
scalar curvature metrics. In \cite{[BaMa]}, Bando-Mabuchi proved
that the $K$-energy is bounded from below on a K\"ahler-Einstein
manifold with $c_1(M)>0$. It has been shown by G. Tian in
\cite{[Tian2]}\cite{[Tian3]} that $M$ admits a K\"ahler-Einstein
metric if and only if  the $K$-energy or $F$ functional defined by Ding-Tian \cite{[DT]} is proper.
Thus, it is natural
to study the relation between $E_k$ functionals and
K\"ahler-Einstein metrics. Following a question posed by Chen in
\cite{[chen1]}, Song-Weinkove studied the lower bound of energy
functionals $E_k$ on K\"ahler-Einstein manifolds. Shortly
afterwards, N. Pali     \cite{[Pali]} gave a formula between $E_1$
and the $K$-energy $E_0$, which implies $E_1$ has a lower bound if the $K$-energy is bounded from below.
Inspired by Song-Weinkove and Pali's work, we proved that the lower boundedness of $F$ functional,
the $K$-energy and $E_1$ are equivalent in the canonical K\"ahler class in \cite{[CLW]}\cite{[L2]}, and
we proved a general formula which gives the relations of all energy functionals $E_k$ in \cite{[L]}.
In \cite{[R]} Y. Rubinstein extended these results and proved all the lower boundedness and properness of $E_k$ functionals
are equivalent under some natural restrictions.

For the case of K\"ahler-Ricci solitons, Tian-Zhu generalized the
$K$-energy and $F$ fucntional  in \cite{[TZ]} and proved that these
generalized energy are bounded from below on a K\"ahler manifold
which admits a K\"ahler-Ricci soliton. In \cite{[CTZ]} Cao-Tian-Zhu
proved the properness of the generalized energy functionals.
Inspired by these work, we will define the generalized Chen-Tian
energy functionals $\td E_k$ in Section 2 and prove the following
result:

\begin{theo}\label{theo1}
Let $(M, \oo)$ be a compact K\"ahler manifold with $c_1(M)>0$ and
$\oo\in 2\pi c_1(M).$ For any $k\in \{0, 1, \cdots, n\}$ we have
\begin{enumerate}
  \item[(a)] If $\td E_k$ is proper on $\cM_{X, k}^+(\oo)$, then $M$ admits a K\"ahler-Ricci soliton with
  respect to $X$;
  \item[(b)] If  $M$ admits a K\"ahler-Ricci soliton $\oo_{KS}$, then $\td E_k$ is proper on $\cM_G(\oo_{KS})'\cap\cM_{X, k}^+(\oo_{KS}).$
  \end{enumerate}
where $\cM_{X, k}^+(\oo_{KS})$ and $\cM_G(\oo_{KS})'$ are some
subspaces of K\"ahler potentials defined in Section 2.
\end{theo}

The idea of the proof is more or less standard. We follow the
continuity method from \cite{[R]} and \cite{[CTZ]} to prove this.
The crucial point is that by the construction of the generalized
energy functionals $\td E_k,$ all the arguments for the
K\"ahler-Einstein case work very well for our situation. Following
the results in \cite{[CLW]} and \cite{[L2]} we discuss the lower
bound of these energy functionals:

\begin{theo}\label{theo2}
Let $(M, \oo)$ be a compact K\"ahler manifold with $c_1(M)>0$ and
$\oo\in 2\pi c_1(M).$ Then for any $k\in \{0, 1, \cdots, n\}$,
$\td E_k$ is bounded from below on $\cM_{X, k}^+(\oo)$ if and only if
 $\td F$ is bounded from below on $\cM_X(\oo).$ Moreover, we have
\beq \inf_{\oo'\in\cM_{X, k}(\oo)}\td E_{k, \oo}(\oo')=
(k+1)\inf_{\oo'\in \cM_X(\oo)} \td F_{ \oo}(\oo')+C_{\oo, X, k}-
\frac {k+1}{V}\int_M\; u_0 e^{\theta_X}\oo^n,\label{eq:e1}\eeq where
$u_0=-h_{\oo}+\theta_X$ and $C_{\oo, X, k}$ is given by \beq C_{\oo,
X, k}=\sum_{i=0}^{k-1}\;(-1)^{k-i}\binom{k+1}{i}\frac 1V\int_M\;
\i\p u_0\wedge \bar\p u_0\wedge (\pbp u_0)^{k-1}\wedge
e^{\theta_X}\oo^{n-k}.\label{eq:c}\eeq
\end{theo}

Here we take the ideas from \cite{[DT]} to prove Theorem
\ref{theo2}. For the energy functionals $E_k$, there are two
different ways to prove their equivalence. In \cite{[CLW]} we use
the K\"ahler-Ricci flow to prove the equivalence of the lower
boundedness of the $K$ energy and $E_1$ energy, and in \cite{[L2]}
we use Perelman's estimates  to prove the equivalence of the $K$
energy and $F$ functional. The flow method is very tricky and we
lack some crucial estimates  here.  In \cite{[R]} Y. Rubinstein
proved that the equivalence of the lower boundedness of the energy
functionals $E_k$ and $F$, which relies on an interesting
observation on the relation of $E_n$ and $F$(cf. Lemma 2.4 in
\cite{[R]}). Here it seems difficult to find such a relation in the
case of generalized energy functionals. Fortunately, we can use the
continuity method in \cite{[DT]} to overcome these difficulties.\\

As a by-product of Theorem \ref{theo2}, we have the following
result:

\begin{theo}\label{theo3}Let $M$ be a compact K\"ahler manifold with $c_1(M)>0$ and
$\oo$ be any given K\"ahler metric in $ 2\pi c_1(M).$ If $\td F$ is bounded from below for the solution $\varphi_t$ of the
equation
\beq (\oo+\pbp\varphi)^n=e^{h_{\oo}-\theta_X(\varphi)-t\varphi},\label{b1}\eeq
then $\td F$ is bounded from below in the class $2\pi
c_1(M).$ Moreover, we have
$$\inf_{\oo'\in 2\pi c_1(M)}\td F_{\oo}(\oo')=\inf_{t\in [0, 1)}\td F_{\oo}(\varphi_t).$$
Similar results also hold for $\td E_k$ on $\cM_{X, k}^+(\oo)$ with $k=0, 1, \cdots, n.$
\end{theo}

This result is inspired by the beautiful work \cite{[Chen4]}. In
\cite{[Chen4]} X. X. Chen proved that if the $K$-energy is bounded
from below and the infimum of the Calabi energy vanishes along a
particular geodesic ray, then the $K$-energy is bounded from below
in the K\"ahler class. As an application, he essentially proved that
for any K\"ahler class admits constant scalar curvature metric, the
$K$-energy in a nearby K\"ahler class with possibly different
complex structure is bounded from below. We remark that under the
assumption of Theorem \ref{theo3}, the solution $\varphi_t$ of
(\ref{b1}) will exist for all $t\in [0, 1)$. Theorem \ref{theo3}
shows that if the $K$-energy is bounded along one solution
$\varphi_t$, then the $K$-energy is bounded from below in the whole
K\"ahler class. It is interesting to know whether there is a similar
phenomenon for the K\"ahler-Ricci flow:

\begin{q}\label{q}If the $K$-energy is bounded from below along a certain
K\"ahler-Ricci flow, is the $K$-energy  bounded from below in
the class $2\pi c_1(M)$?

\end{q}

In \cite{[TZ]}, Tian-Zhu introduced a new holomorphic invariant $\cF_X(\cdot)$ from the space
of holomorphic vector fields $\eta(M)$ into $\CC$:
\beq
\cF_X(Y)=\int_M\; Y(h_g-\theta_X(g)) e^{\theta_X(g)}\,\oo_g^n,\quad Y\in \eta(M).\n
\eeq
The invariant $\cF_X(\cdot)$ is defined for any holomorphic vector fields $X\in \eta(M)$ and it is independent
of the choice of $g$ with the K\"ahler class $\oo_g\in 2\pi c_1(M).$ When $X=0$, $\cF_X(\cdot)$
is exactly the Futaki invariant. By the definition, we see that $\cF_X$ is an obstruction to the
existence of K\"ahler-Ricci solitons.  The next result shows that the holomorphic invariants defined by $\td E_k$ are scalar multiples of $\cF_X(\cdot)$,
which generalized the results for energy functionals $E_k$(cf. \cite{[Liu]}\cite{[L]}).

\begin{theo} Let $Y$ be a holomorphic vector field and $\{\Phi(t)\}_{|t|<\infty}$ the one-parameter
subgroup of automorphisms induced by $Re(Y),$ we have
$$\frac d{dt}\td E_k(\varphi)=\frac {(k+1)n}V\cF_X(Y),$$
where $\varphi$ is given by $\Phi^*_t\oo=\oo+\pbp\varphi.$
\end{theo}

{\bf Acknowledgements}: The author  would like to thank  Professor X. X. Chen, W. Y. Ding and F. Pacard for their constant,
warm encouragements over the past several years.

\section{Energy functionals}
In this section, we recall some energy functionals introduced by
Tian-Zhu in \cite{[TZ]} and give the definition of the generalized
Chen-Tian energy function.

 Let $M $ be an $n$-dimensional compact
K\"ahler manifold with positive first Chern class, and $\oo$ be a
fixed K\"ahler metric in the K\"ahler class $2\pi c_1(M).$ Then
there is a smooth real-valued function $h_{\oo}$ such that
$$Ric(\oo)-\oo=\pbp h_{\oo}, \quad \int_M\; (e^{h_{\oo}}-1)\oo^n=0.$$
Suppose that $X$ is a holomorphic vector field on $M$ so that the
integral curve of $K_X$ of the imaginary part $Im(X)$ of $X$
consists of isometries of $\oo$. By the Hodge decomposition theorem,
there exists a unique smooth real-valued function $\theta_X$ on $M$
such that
$$i_X\oo=\i \bar\partial \theta_X, \quad \int_M\; (e^{\theta_X}-1)\oo^n=0.$$
Now we define the space of K\"ahler potentials which are invariant
under $Im(X):$
$$\cM_X(\oo)=\{\varphi\in C^{\infty}(M, \RR)\;|\; \oo_{\varphi}=\oo+\pbp \varphi>0,\;\; Im(X)(\varphi)=0\;\}.$$
 Tian-Zhu in \cite{[TZ]} introduced the
following functional, which can be regarded as a generalization of
Mabuchi's $K$ energy, \beq \td E_{0, \oo}(\varphi)=\frac
{n}{V}\int_0^1\int_M\;\i\p \pd {\varphi_t}t\wedge \bar \p
(h_{\varphi_t}-\theta_X(\varphi_t))\wedge e^{\theta_X(\varphi_t)}
\oo_{\varphi_t}^{n-1}\wedge dt\label{eq:K energy}\eeq
 where
$\varphi_t(t\in [0, 1])$ is a path connecting $0$ and $\varphi$ in
$\cM_X(\oo)$ and $$\theta_X(\varphi_t)=\theta_X+X(\varphi_t)$$ is
the potential function of $X$ with respect to the metric
$\oo_{\varphi_t}.$ We define the following functionals on
$\cM_X(\oo)$:
$$\td I_{\oo}(\varphi)=\frac 1V\int_M\; \varphi(e^{\theta_X}\oo^n-e^{\theta_X(\varphi)}\oo_{\varphi}^n),$$
and
$$\tilde J_{\oo}(\varphi)=\frac 1V\int_0^1\;\int_M\; \pd {\varphi_t}{t}(e^{\theta_X}\oo^n-e^{\theta_X(\varphi_t)}
 \oo_{\varphi_t}^n
)\wedge dt.$$ As before, $\varphi_t(t\in [0, 1])$ is a path
connecting $0$ and $\varphi$ in $\cM_X(\oo).$ Then for any path
$\varphi_t$, we have
$$\frac d{dt}(\td I_{\oo}(\varphi_t)-\td J_{\oo}(\varphi_t))=
-\frac 1V\int_M\;\pd {\varphi}t(\Delta_{\varphi_t}+X)\varphi_t\,
e^{\theta_X(\varphi_t)}\oo_{\varphi_t}^n.$$ Then, by Lemma 3.1 in
\cite{[CTZ]} there exist two positive constants $c_1(n)$ and
$c_2(n)$ such that \beq c_1I_{\oo}(\varphi)\leq \td
I_{\oo}(\varphi)-\td J_{\oo}(\varphi)\leq c_2
I_{\oo}(\varphi),\label{I inequality}\eeq where \beq
I_{\oo}(\varphi)=\frac
1V\int_M\;\varphi(\oo^n-\oo_{\varphi}^n).\label{I1}\eeq In
\cite{[TZ]} Tian-Zhu defined the generalized $F$ functional which is
defined by Ding-Tian in \cite{[DT]} as follows \beq \tilde
F_{\oo}(\varphi)=\tilde J_{\oo}(\varphi)-\frac 1V\int_M\;\varphi
e^{\theta_X} \oo^n-\log \Big(\frac 1V\int_M\;
e^{h_{\oo}-\varphi}\oo^n\Big).\label{defi F}\eeq The generalized $F$
functional has exactly the same behavior as in the K\"ahler-Einstein
case. For example, the generalized $K$-energy and $\td F$ functional
are related by the identity (cf. \cite{[TZ]}) \beq \td
E_0(\varphi)=\td F(\varphi)+ \frac 1V\int_M\; u
e^{\theta_X(\varphi)}\oo_{\varphi}^n-\frac 1V\int_M\; u_0
e^{\theta_X}\oo^n+\log\Big(\frac 1V\int_M\;
e^{h_{\oo}-\varphi}\oo^n\Big), \label{eq:K F}\eeq where $u$ is
defined by
 \beq
u(\varphi)=-h_{\varphi}+\theta_X(\varphi)=\log\frac
{\oo_{\varphi}^n}{\oo^n}+\varphi
-h_{\oo}+\theta_X(\varphi),\label{u:defi}\eeq and
$u_0=-h_{\oo}+\theta_X$. It follows  that $\td E_0$ is always bigger
than $\td F$ up to a constant: \beq \td E_0(\varphi)\geq\td
F(\varphi)-\frac 1V\int_M\; u_0 e^{\theta_X}\oo^n. \label{ineq:E0
F}\eeq

Now we recall some  results in \cite{[CTZ]}.
 Let $K_0(\supset K_X)$ be a maximum
compact subgroup of the automorphisms group of $M$ such that
$\si\cdot \eta=\eta\cdot \si$ for any $\eta\in K_0$ and any $\si\in
K_X.$ If $\oo_{KS}$ is a K\"ahler-Ricci soliton with respect to the
holomorphic vector field $X$, we define the inner product by
$$(\varphi, \psi)=\int_M\; \varphi\psi e^{\theta_X(\oo_{KS})}\oo_{KS}^n,$$
and denote by
$$\La_1(\oo_{KS})=\{u\in C^{\infty}\;|\; \Delta_{KS}u+X(u)=-u\}.$$
For any compact subgroup $G\supset K_X$ of $K_0$ with  $\si\cdot
\eta=\eta\cdot \si$ for any $\eta\in G$ and any $\si\in K_X,$ we
denote by $\cM_G(\oo_{KS})'$ the space of $G$-invariant K\"ahler
potentials perpendicular to $\La_1(\oo_{KS}).$ We call a functional
$F(\varphi)$  proper on $\cM_X(\oo)$, if there exists an increasing
function $\rho: \RR\ri \RR$ satisfying $\lim_{t\ri
+\infty}\rho(t)=+\infty$ such that for any $\varphi\in \cM_X(\oo)$,
$F(\varphi)\geq \rho(I_{\oo}(\varphi))$, where $I_{\oo}(\varphi)$ is
given by (\ref{I1}).

 In
\cite{[CTZ]}, Cao-Tian-Zhu proved the following result, which is
crucial in the proof of Theorem \ref{theo1}.

\begin{theo}(cf. \cite{[CTZ]})\label{theo:CTZ} If $M$ admits a K\"ahler-Ricci soliton, then $\td F$ is proper
on $\cM_G(\oo_{KS})'$.
\end{theo}

Inspired by the work in \cite{[TZ]} and  \cite{[L]}, we
define the generalized Chen-Tian energy functionals $\tilde E_k$ as
follows:
\medskip

\begin{defi}\label{defi:E}
We define the generalized Chen-Tian energy functionals for any $k=1,
2, \cdots, n,$
$$\tilde
E_{k, \oo}(\varphi)=\sum_{i=0}^{k-1}\;(-1)^{k-i}\binom{k+1}{i}\tilde
G_{k-i, \oo}(\varphi)+(k+1)\tilde E_0(\varphi),$$ where \beqs \tilde
G_{k, \oo}(\varphi) &=&-\frac 1V\int_M\; \i\p u\wedge \bar\p u\wedge
(\pbp
u)^{k-1}\wedge e^{\theta_X(\varphi)}\oo_{\varphi}^{n-k}\\
&&+\frac 1V\int_M\; \i\p u_0\wedge \bar\p u_0\wedge (\pbp
u_0)^{k-1}\wedge e^{\theta_X}\oo^{n-k}, \eeqs
where $u=-h_{\varphi}+\theta_X(\varphi)$ and $u_0=-h_{\oo}+\theta_X.$
\end{defi}

\begin{rem}For Chen-Tian energy functionals $E_k$, there are many
different expressions as in \cite{[chen-tian1]} \cite{[SoWe]} and
\cite{[R]}. It is interesting how to write the generalized
functionals $\td E_k$ as similar expressions.

\end{rem}

By the definition, it is easy to check that all of $\td E_k$ satisfy
the following cocycle condition
$$\td E_{k, \oo}(\varphi)+\td E_{k, \oo_{\varphi}}(\psi-\varphi)=\td E_{k, \oo}(\psi),$$
for any $\varphi, \psi \in \cM_X(\oo).$ Let $k=1,$  we have the
generalized Pali's formula: \beqn \td E_1(\varphi)&=&2\td
E_0(\varphi)+\frac 1V\int_M\; \i \p u\wedge \bar \p u\wedge
e^{\theta_X(\varphi)}\oo^{n-1}_{\varphi} -\frac 1V\int_M\; \i \p
u_0\wedge \bar \p u_0\wedge
e^{\theta_X}\oo^{n-1}\n\\
&\geq &2\td E_0(\varphi)-C_{\oo, X}.\label{eq:E0 E_1} \eeqn Now we
define the subspace of K\"ahler potentials for $k=2, 3, \cdots, n$
\beq \cM_{X, k}^+(\oo)=\{\varphi\in
\cM_X(\oo)\;|\;Ric_{\varphi}-L_X\oo_{\varphi}\geq -\frac
2{k-1}\oo_{\varphi}\},\label{M defi}\eeq and let $\cM_{X,
0}^+(\oo)=\cM_{X, 1}^+(\oo)=\cM_{X}(\oo).$ The definition of (\ref{M
defi}) is inspired by the result in \cite{[L]}. With these
notations, we have the result:

\begin{lem}\label{lem:Ek}For any $\oo_{\varphi}\in \cM_{X, k}^+(\oo)(k\geq 2),$ we have
$$\td E_k(\varphi)\geq (k+1)\td E_0(\varphi)-C_{\oo, X, k},$$
where $C_{\oo, X, k}$ is given by (\ref{eq:c}).
\end{lem}
\begin{proof}The argument is the same as in \cite{[L]} and here we give the
details for completeness.   By the definition of $\td E_k$, we have
\beqs &&\td E_k-(k+1)\td
E_0\\&=&\sum_{i=0}^{k-1}\;(-1)^{k-i}\binom{k+1}{i}\td G_{k-i}\\&=&
\frac 1V\int_M\;\sqrt{-1}\partial u\wedge\bar \partial u\wedge
\Big(\sum_{i=0}^{k-1}(-1)^{k-i-1}\binom{k+1}{i}(\pbp
u)^{k-i-1}\wedge \oo_{\varphi}^i\Big)\wedge
e^{\theta_X(\varphi)}\oo_{\varphi}^{n-k}+C_{\oo, X, k}
\\&=& \frac
1V\int_M\;\sqrt{-1}\partial u\wedge\bar \partial u\wedge
\Big(\sum_{i=0}^{k-1}\binom{k+1}{i}(Ric_{\varphi}-L_X\oo_{\varphi}-\oo_{\varphi})^{k-i-1}\wedge
\oo_{\varphi}^i\Big)\wedge
e^{\theta_X(\varphi)}\oo_{\varphi}^{n-k}+C_{\oo, X, k}. \eeqs
Observe that \beq
\sum_{i=0}^{k-1}\binom{k+1}{i}(Ric_{\varphi}-L_X\oo_{\varphi}-\oo_{\varphi})^{k-i-1}\wedge
\oo_{\varphi}^i=\sum_{i=1}^{k}i
(Ric_{\varphi}-L_X\oo_{\varphi})^{k-i}\wedge
\oo_{\varphi}^{i-1}.\label{x1}\eeq For $k\geq 2,$ let
$$P(x)=\sum_{i=1}^{k}ix^{k-i}=(x+\frac
2{k-1})^{k-1}+\sum_{i=2}^k\;a_i (x+\frac 2{k-1})^{k-i},$$ where
$a_i$ are the constants defined by
$$a_i=\frac 1{(k-i)!}P^{(k-i)}(-\frac 2{k-1}).$$
By Lemma {A.1} in the appendix of \cite{[L]}, $a_i\geq 0.$ Thus, for any $\varphi\in\cM_{X, k}^+(\oo) $ we have
\beqs &&\sum_{i=1}^{k}i (Ric_{\varphi}-L_X\oo_{\varphi})^{k-i}\wedge
\oo_{\varphi}^{i-1}\\&=&\Big(Ric_{\varphi}-L_X\oo_{\varphi}+\frac
2{k-1}\oo_{\varphi}\Big)^{k-1}+\sum_{i=2}^k\;a_i
\Big(Ric_{\varphi}-L_X\oo_{\varphi}+\frac 2{k-1}\oo_{\varphi}\Big)^{k-i}\wedge
\oo_{\varphi}^{i-1}\geq 0.
\eeqs
 Therefore, $\td E_k\geq (k+1)\td E_0+C_{\oo, X, k}$ and the lemma
 is proved.

\end{proof}

\section{Proof of Theorem \ref{theo1}}

Suppose that $M$ admits a K\"ahler-Ricci soliton in the class
$\cM_X(\oo)$. This implies that $\td F$ functional is proper on
$\cM_G(\oo_{KS})'$ by Theorem \ref{theo:CTZ}, and also $\td E_0$ is
 proper on $\cM_G(\oo_{KS})'$ by (\ref{ineq:E0 F}). Thus,  $\td
E_1$ is also proper on $\cM_G(\oo_{KS})'$
 by (\ref{eq:E0 E_1})  and so is $\td E_k$ on $\cM_{X, k}^+(\oo)\cap \cM_G(\oo_{KS})'$  for any $k\in \{2, \cdots,
n\}$ by Lemma \ref{lem:Ek}. Thus, part $(b)$ of Theorem \ref{theo1}
is proved.

To finish part $(a)$ of Theorem \ref{theo1}, it suffices  to prove:

\begin{lem}\label{lem:theo1}
If $\td E_k$ is proper on $ \cM_{X, k}^+(\oo)$ for any $k\in \{2, 3,
\cdots, n\}$, then there exists a K\"ahler-Ricci soliton on $M.$
\end{lem}
\begin{proof}We consider the complex Monge-Ampere equations with parameter $t\in
[0, 1]$ \beq (\oo+\pbp \varphi)^n=
e^{h-\theta_X(\varphi)-t\varphi}\oo^n. \label{eq soliton} \eeq There
exists a unique solution at $t=0$ modulo constants by \cite{[Z]},
and the set of $t\in [0, 1]$ such that (\ref{eq soliton}) has a
solution is open by the implicit function theorem(cf. \cite{[TZ2]}).
 Therefore, to prove that there is a
solution for $t=1,$ it suffices to prove that $ I_{\oo}(\varphi)$ is
uniformly bounded for $0\leq t< 1.$

Note that the solution
$\varphi_t\in \cM_{X, k}^+(\oo)$, since the equation (\ref{eq soliton}) can be written as
$$Ric_{\varphi}-L_X\varphi=t\oo_{\varphi}+(1-t)\oo>0.$$
Since $\td E_k$ is proper on $ \cM_{X, k}^+(\oo)$, there exists an increasing function $\rho: \RR\ri \RR$
satisfying $\lim_{s\ri +\infty}\rho(s)=+\infty$ such that
$\td E_k(\varphi(t))\geq \rho(I_{\oo}(\varphi(t))).$ Now we show that $\td E_k$ is uniformly bounded
from above for $t\in [0, 1).$ In fact, \beqs \pd {}t\td
E_0(\varphi_t)&=&\frac nV\int_M\;\i\p \pd {\varphi}t\wedge \bar \p
(h_{\varphi}-\theta_X(\varphi))\wedge e^{\theta_X(\varphi)}
\oo_{\varphi}^{n-1}\\
&=&-\frac nV\int_M\;\i\p \pd {\varphi}t\wedge \bar \p u\wedge
e^{\theta_X(\varphi)} \oo_{\varphi}^{n-1}. \eeqs Thus, for the
solution $\varphi_t(0\leq t\leq \tau\leq 1)$ we have
\beqn
\td E_0(\varphi_{\tau})-\td E_0({\varphi_0})&=&-\frac nV\int_0^{\tau}\int_M\;(1-t)\i
\p \pd {\varphi}t\wedge \bar \p \varphi\wedge
e^{\theta_X(\varphi)} \oo_{\varphi}^{n-1}\wedge dt\n\\
&=&\frac 1V\int_0^{\tau}\int_M\;(1-t)\varphi(\Delta_{\varphi}+X)\pd {\varphi}t
e^{\theta_X(\varphi)} \oo_{\varphi}^{n}\wedge dt\n\\
&=&-\int_0^{\tau}\;(1-t)\frac d{dt}(\td I-\td J)dt\n\\
&=&-(1-\tau)(\td I-\td J)(\varphi_{\tau})+(\td I-\td J)(\varphi_0)-\int_0^{\tau}(\td I-\td J)dt \label{eq:c1}\\
&\leq&-c(n)(1-\tau)I_{\oo}(\varphi_{\tau})+(\td I-\td J)(\varphi_0)-\int_0^{\tau}(\td I-\td J)dt,\n
\eeqn
where we have used the inequality (\ref{I inequality}).
Hence, by the definition of $\td E_k$ we have
\beqn
\td E_k(\varphi_{\tau})&=&(k+1)E_0(\varphi_{\tau})+\sum_{i=0}^{k-1}\;(-1)^{k-i}\binom{k+1}{i}\tilde
G_{k-i, \oo}(\varphi_{\tau})\n\\
&\leq&(k+1)\td E_0(\varphi_0) -c(n)(k+1)(1-\tau)I_{\oo}(\varphi_{\tau})+(k+1)(\td I-\td J)(\varphi_0)\n\\
&&-(k+1)\int_0^{\tau}(\td I-\td J)dt+\sum_{i=0}^{k-1}\;(-1)^{k-i}\binom{k+1}{i}\tilde
G_{k-i, \oo}(\varphi_{\tau}). \label{a1}
\eeqn
Note that $\varphi_{\tau}$ satisfies the equation (\ref{eq soliton}) and we have
\beqs u(\tau)&=&\log\frac {\oo_{\varphi_{\tau}}^n}{\oo^n}+\varphi_{\tau}-h_{\oo}+\theta_X(\varphi_{\tau})\\
&=&(1-\tau)\varphi_{\tau}.\eeqs
Thus, we have
\beqn
&&\sum_{i=0}^{k-1}\;(-1)^{k-i}\binom{k+1}{i}\tilde
G_{k-i, \oo}(\varphi_{\tau})\n\\&=&\frac 1V\int_M\;\i \p u\wedge\bar \p u\wedge \Big(\sum_{i=0}^{k-1}(-1)^{k-i-1}
\binom{k+1}{i} (\pbp u)^{k-i-1}\wedge \oo_{\varphi_{\tau}}^i\Big)\wedge e^{\theta_X(\varphi_{\tau})}\oo_{\varphi_{\tau}}^{n-k}+C_{\oo, X, k}\n\\
&=&\frac 1V\int_M\;(1-\tau)^2\i \p \varphi_{\tau}\wedge\bar \p \varphi_{\tau}\wedge \n\\
&&\Big(\sum_{i=0}^{k-1}(-1)^{k-i-1}
\binom{k+1}{i} ((1-\tau)\pbp \varphi_{\tau})^{k-i-1}\wedge \oo_{\varphi_{\tau}}^i\Big)\wedge e^{\theta_X(\varphi_{\tau})}\oo_{\varphi_{\tau}}^{n-k}+C_{\oo, X, k}\n\\
&=&\frac 1V\int_M\;(1-\tau)^2\i \p \varphi_{\tau}\wedge\bar \p \varphi_{\tau}\wedge \Big(\sum_{i=1}^k i
(\tau \oo_{\varphi_{\tau}}+(1-\tau)\oo)^{k-i}\wedge \oo_{\varphi_{\tau}}^{i-1}\Big)\wedge e^{\theta_X(\varphi_{\tau})}\oo_{\varphi_{\tau}}^{n-k}+C_{\oo, X, k}\n\\
&\leq&(1-\tau)^2c(n)I(\varphi_{\tau})+C_{\oo, X, k}.\label{a2} \eeqn
Combining this with inequality (\ref{a1}),  for any $\tau$
sufficiently close to $1$ we have
$$\rho(I(\varphi_{\tau}))\leq \td E_k(\varphi_{\tau})\leq C(\oo, \varphi_0).$$
Hence $I(\varphi_t)$ is uniformly bounded from above for $t\in [0, 1)$. Thus, $|\varphi_t|_{C^0}$ and all higher order estimates are uniformly
bounded for any $t\in [0, 1)$ and the solution $\varphi_t(t\in [0, 1))$ can be extended to $t=1$ smoothly. This concludes that
$M$ admits a K\"ahler-Ricci soliton.

\end{proof}

\section{Proof of Theorem \ref{theo2}}
Suppose $\td F$ is bounded from below on $\cM_X(\oo)$. Then by
(\ref{ineq:E0 F}) $\td E_0$ is bounded from below and so is $\td
E_k$ on $\cM_k^+(\oo)$ by Lemma \ref{lem:Ek} for any $k\in \{1, 2,
\cdots, n\}.$ Thus, it suffices to prove the following

\begin{lem}If $\td E_k$ is bounded from below on $\cM_k^+(\oo)$ for any $k\in \{1, 2,
\cdots, n\},$ then $\td F$ is bounded from below on $\cM(\oo).$
\end{lem}
\begin{proof}For any $\psi\in \cM_X(\oo),$ we set $\oo_s=\oo+s\pbp
\psi,$ and let $\varphi_{s, t}$ be the solution of the equation
\beq
(\oo_s+\pbp \varphi)^n=e^{h_s-\theta_{s}(\varphi)-t\varphi}\oo_s^n,\label{theo2:eq1}
\eeq
where $h_s$  satisfies
\beq Ric(\oo_s)-\oo_s=\pbp h_s, \quad \int_M\; e^{h_s} \oo_s^n=V,\label{lem:theo2:eq3}\eeq
and $\theta_s$ is defined by
$$\theta_s=\theta_X+X(s\psi),\quad \int_M\; e^{\theta_s}\oo_s^n=V. $$
Since $\td E_k$ is bounded from below, from the proof of Lemma \ref{lem:theo1} the solution $\varphi_{s, t}$ of (\ref{theo2:eq1}) exists for any $t\in [0, 1)$ and each $s\in [0, 1].$
Now we have the following

\begin{claim}\label{lem:theo2}For any $s\in [0, 1]$ we have
\beqn
-\infty<\lim_{t\ri 1^-}\td F_{\oo_s}(\varphi_{s, t})\leq 0. \label{lem:theo2:eq1}
\eeqn

\end{claim}
\begin{proof}By the proof of Lemma \ref{lem:theo1}, for any $t\in [0, 1)$ we have
\beqn
\td E_{k, \oo_s}(\varphi_{s, t}) \leq -c(n, k)(1-t)I_{\oo}(\varphi_{s, t})+C(\oo_s)-(k+1)\int_0^{t}(\td I_{\oo_s}(\varphi_{s, \tau})-\td J_{\oo_s}(\varphi_{s, \tau}))d\tau.\n
\eeqn
By the assumption that $\td E_k$ is bounded from below, we have
\beq
c(n, k)(1-t)I_{\oo}(\varphi_{s, t})+(k+1)\int_0^{1}(\td I_{\oo_s}(\varphi_{s, \tau})-\td J_{\oo_s}(\varphi_{s, \tau}))d\tau\leq C(\oo_s).\label{lem:theo2:eq2}
\eeq
Note that $\td I_{\oo_s}(\varphi_{s, \tau})-\td J_{\oo_s}(\varphi_{s, \tau})$ is increasing
with respect to $\tau,$ we have
\beq
0\leq \td I_{\oo_s}(\varphi_{s, t})-\td J_{\oo_s}(\varphi_{s, t})\leq
\frac 1{1-t}\int_t^1(\td I_{\oo_s}(\varphi_{s, \tau})-\td J_{\oo_s}(\varphi_{s, \tau}))d\tau.\n
\eeq
and
\beq \lim_{t\ri 1^-}(1-t)(\td I_{\oo_s}(\varphi_{s, t})-\td J_{\oo_s}(\varphi_{s, t}))=0. \label{theo2:eq2}\eeq

By Proposition 3.1 in \cite{[CTZ]}, there exists two constants $c_1=c_1(X, \oo)$ and $c_2=c_2(X, \oo)$ such that for any $t\in [\frac 12, 1)$
\beq
\|\varphi_{s, t}\|_{C^0}\leq c_1 I_{\oo_s}(\varphi_{s, t})+c_2.\n
\eeq
Combining this with (\ref{theo2:eq2})(\ref{I inequality}), for any $s\in [0, 1]$ we have
\beq
\lim_{t\ri 1^-}(1-t)\|\varphi_{s, t}\|_{C^0}= 0. \label{theo2:eq4}
\eeq
Note that
$$\frac d{dt}\int_M\; e^{\theta_s(\varphi_{s, t})}\oo_{\varphi_{s, t}}^n=\int_M\;
 (\Delta_{s, t}+X)\pd {\varphi_{s, t}}te^{\theta_s(\varphi_{s, t})}\oo_{\varphi_{s, t}}^n=0,$$
we infer that
$$\int_M\; e^{h_{s}-t\varphi_{s, t}}\oo^n=\int_M\; e^{\theta_s(\varphi_{s, t})}
\oo_{\varphi_{s, t}}^n=\int_M\; e^{\theta_s}\oo_s^n=V.$$
Combining this with (\ref{theo2:eq4}), for any $s\in [0, 1]$ we have
\beq
\lim_{t\ri 1^-}\int_M\; e^{h_{s}-\varphi_{s, t}}\oo_s^n=\lim_{t\ri 1^-}\int_M\; e^{h_{s}-t\varphi_{s, t}}\cdot e^{(t-1)\varphi_{s, t}}\oo_s^n=V,\label{lem:theo2:eq5}
\eeq
and we can infer that
\beqn
\lim_{t\ri 1^-}\td F_{\oo_s}(\varphi_{s, t})&=&\lim_{t\ri 1^-}\Big(
\td J_{\oo_s}( \varphi_{s, t})-\frac 1V\int_M\;  \varphi_{s, t}e^{\theta_s}\oo_s^n\Big)\n\\
&=&-\int_0^{1}(\td I_{\oo_s}(\varphi_{s, \tau})-\td J_{\oo_s}(\varphi_{s, \tau}))d\tau\leq 0,\label{eq:f1}
\eeqn where we used (\ref{lem:theo2:eq2}) and the fact that (cf. Proposition 1.1 in \cite{[CTZ]})
$$\td J_{\oo_s}( \varphi_{s, t})-\frac 1V\int_M\;  \varphi_{s, t}e^{\theta_s}\oo_s^n=-\frac 1{t}\int_0^t\;
(\td I_{\oo_s}(\varphi_{s, \tau})-\td J_{\oo_s}(\varphi_{s, \tau}))d\tau.$$
Thus, the claim is proved.
\end{proof}
\medskip

\begin{claim}\label{lem2:theo2}For any $s\in [0, 1]$, we have
$$\lim_{t\ri 1^-}\td F_{\oo_s}(\varphi_{s, t})=\lim_{t\ri 1^-}\td F_{\oo_s}(\varphi_{0, t}).$$
In other words, the limit $\lim_{t\ri 1^-}\td F_{\oo_s}(\varphi_{s, t})$ is independent of $s$.

\end{claim}
\begin{proof}By (\ref{lem:theo2:eq3}) we have
\beq
h_s=-\log\frac {\oo_s^n}{\oo^n}-s\psi+h_{\oo}+c_s\label{lem2:theo2:eq1}
\eeq
where $c_s$ is a constant given by
\beq
\int_M\; e^{h_{\oo}-s\psi+c_s}\oo^n=V.\n
\eeq
Thus, (\ref{theo2:eq1}) can be written as
\beq
(\oo+\pbp (s\psi+\varphi_{s, t}))^n=e^{h_{\oo}-\theta_{X}-
X(s\psi+\varphi_{s, t})-t\varphi_{s, t}-s\psi+c_s}\oo^n.\n
\eeq
Let $\hat \varphi_{s, t}=s\psi+\varphi_{s, t}-c_s$, we have
\beq
(\oo+\pbp \hat \varphi_{s, t})^n=e^{h_{\oo}-\theta_X(\hat \varphi_{s, t})-t\hat \varphi_{s, t}-(1-t)(s\psi-c_s)}\oo^n.
\eeq
Taking derivative with respect to $s$, we have
\beq
(\Delta_{s, t}+X)\pd {\hat \varphi_{s, t}}s=-t\pd {\hat \varphi_{s, t}}s-(1-t)(\psi-\frac {dc_s}{ds}).\label{lem:theo2:eq4}
\eeq
Direct calculation shows that
\beqn
\pd {}s\Big(\td J_{\oo}(\hat \varphi_{s, t})-\frac 1V\int_M\; \hat \varphi_{s, t}e^{\theta_X}\oo^n\Big)&=&
-\frac 1V\int_M\; \pd {\hat \varphi_{s, t}}s e^{\theta_X(\hat \varphi_{s, t})}\oo_{s, t}^n\n\\
&=&-\frac {1-t}{tV}\int_M\; (\psi-\frac {dc_s}{ds})\oo^n_{s, t}, \label{lem2:theo2:eq3}
\eeqn where we used (\ref{lem:theo2:eq4}).
Note that for any $s\in [0, 1]$ we have
\beq \lim_{t\ri 1^-}\int_M\; e^{h_{\oo}-\hat \varphi_{s, t}}\oo^n=\lim_{t\ri 1^-}\int_M\; e^{h_s-\varphi_{s, t}}\oo_s^n=V,\label{lem2:theo2:eq2}\eeq
where we used (\ref{lem:theo2:eq5}) and (\ref{lem2:theo2:eq1}). Combining (\ref{lem2:theo2:eq3}) with (\ref{lem2:theo2:eq2}), for any $s\in [0, 1]$ we have
\beqn
\lim_{t\ri 1^-}\Big|\td F_{\oo}(\hat \varphi_{s, t})-\td F_{\oo}(\hat \varphi_{0, t})\Big|&=&
\lim_{t\ri 1^-}\Big|\int_0^s\pd {}{\tau}\Big(\td J_{\oo}(\hat \varphi_{\tau, t})-\frac 1V\int_M\; \hat \varphi_{\tau, t}e^{\theta_X}\oo^n\Big)d\tau\Big|\n
\\&=&\lim_{t\ri 1^-}\Big|\frac {1-t}{tV}\int_0^s\,d\tau\int_M\; (\psi-\frac {dc_{\tau}}{ds})\oo^n_{\tau, t}\Big|\n\\
&=&0.\n
\eeqn
The claim is proved.

\end{proof}
By Claim \ref{lem:theo2} we have
\beq
\lim_{t\ri 1^-}(\td F_{\oo}(\oo_{s, t})-\td F_{\oo}(\oo_s))=\lim_{t\ri 1^-}\td F_{\oo_s}(\oo_{s, t})\leq 0.\n
\eeq
and by Claim \ref{lem2:theo2}
\beq
\td F_{\oo}(\psi)=\td F_{\oo}(\oo_1)\geq\lim_{t\ri 1^-}\td F_{\oo}(\oo_{1, t})=\lim_{t\ri 1^-}\td F_{\oo}(\oo_{0, t}).\n
\eeq
Thus, $\td F$ is uniformly bounded from below on $\cM_X(\oo).$

\bigskip

Now we prove the equality (\ref{eq:e1}). Suppose that one of the energy functionals $\td E_k$  and $\td F$
is bounded from below on $\cM_{X, k}(\oo)$,  by Lemma \ref{lem:Ek} and the inequality (\ref{ineq:E0 F})
we have \beqn
\inf_{\oo'\in\cM_{X, k}(\oo)}\td E_{k, \oo}(\oo')&\geq&
(k+1)\inf_{\oo'\in \cM_X(\oo)}\td E_{0, \oo}(\oo')+C_{\oo, X, k}\n\\
&\geq&(k+1)\inf_{\oo'\in \cM_X(\oo)}\td F_{\oo}(\oo')+C_{\oo, X,
k}-\frac {k+1}V\int_M\; u_0e^{\theta_X}\oo^n. \label{eq:d1} \eeqn On
the other hand, for the solution $\varphi_t(t\in [0, 1))$ of
(\ref{eq soliton}) the inequality (\ref{a2}) implies that
$$\lim_{t\ri 1^-}\sum_{i=0}^{k-1}\;(-1)^{k-i}\binom{k+1}{i}\tilde
G_{k-i, \oo}(\varphi_t)\leq C_{\oo, X, k}.$$ Combining this with the
definition of $\td E_k$ and the equality (\ref{eq:c1}) (\ref{eq:f1})
we have \beqn \inf_{\oo'\in\cM_{X, k}(\oo)}\td E_{k,
\oo}(\oo')&\leq& \lim_{t\ri 1^-}\td E_{k, \oo}(\varphi_t)\leq
(k+1)\lim_{t\ri 1^-}\td E_{0, \oo}(\varphi_t)+C_{\oo, X, k}\n
\\&=&(k+1)\Big(\td E_{0, \oo}(\varphi_0)+(\td I-\td J)_{\oo}(\varphi_0)-
\int_0^1(\td I-\td J)_{\oo}(\varphi_{\tau})d\tau\Big)+C_{\oo, X, k}\n\\
&=&(k+1)\lim_{t\ri 1^-}\td F_{\oo}(\varphi_t)+C_{\oo, X, k}-\frac {k+1}V\int_M\; u_0e^{\theta_X}\oo^n\n\\
&=&(k+1)\inf_{\oo'\in \cM_X(\oo)}\td F_{\oo}(\oo')+C_{\oo, X,
k}-\frac {k+1}V\int_M\; u_0e^{\theta_X}\oo^n, \label{eq:d2} \eeqn
where $\varphi_t(t\in [0, 1))$ is the solution of (\ref{eq soliton})
and we have used
$$\td E_{0, \oo}(\varphi_0)+(\td I-\td J)_{\oo}(\varphi_0)=-\frac 1V\int_M\; u_0e^{\theta_X}\oo^n$$
since $\varphi_0$ is a solution of (\ref{eq soliton}) when $t=0.$
Combining (\ref{eq:d1})(\ref{eq:d2}), we have \beq
\inf_{\oo'\in\cM_{X, k}(\oo)}\td E_{k, \oo}(\oo')=
(k+1)\inf_{\oo'\in \cM_X(\oo)}\td F_{\oo}(\oo')+C_{\oo, X, k}-\frac
{k+1}V\int_M\; u_0e^{\theta_X}\oo^n.\n\eeq The theorem is proved.
\end{proof}

Following the ideas of the previous proof, we can finish Theorem
\ref{theo3}:

\begin{proof}[Proof of Theorem \ref{theo3}] For any $\psi\in
\cM_X(\oo),$ we consider the solution $\varphi_{s, t}$ of the
equation (\ref{theo2:eq1}). Suppose that  $\td F$ is bounded
from below for the solution $\varphi_{0, t}$, then $\varphi_{0, t}$
exists for all $t\in [0, 1)$ and $$\lim_{t\ri 1^-}\td
F_{\oo}(\varphi_{0, t})=-\int_0^1\;(\td I_{\oo}-\td
J_{\oo})(\varphi_{0, \tau})d\tau>-\infty.$$ For simplicity, we set
$$\td F^0_{\oo}(\varphi)=\td J_{\oo}(\varphi)-\frac 1V\int_M\; \varphi \,e^{\theta_X}\oo^n.$$
Then by (\ref{lem2:theo2:eq3}) for any $s\in [0, 1]$ we have \beqn
\td F_{\oo}^0(\hat \varphi_{s, t})-\td F_{\oo}^0(\hat \varphi_{0, t})=-\int_0^s\;\frac {1-t}{t V}\int_M\;
\Big(\psi-\frac {dc_{\tau}}{d\tau}\Big)\oo_{\tau, t}^nd\tau, \label{theo3:eq2}
 \eeqn
which implies that
$$\td F^0_{\oo_s}(\varphi_{s, t})=\td F^0_{\oo}(\hat \varphi_{s, t})-\td F^0_{\oo}(s\psi)$$
is uniformly bounded from below for any $t\in [\frac 12, 1)$ and $s\in [0, 1]$. Thus, the solution $\varphi_{s, t}$ exists for all
$t\in [0, 1)$ when $s\in [0, 1]$ and
$$\lim_{t\ri 1^-}\td F^0_{\oo_s}(\varphi_{s, t})=\lim_{t\ri 1^-}\td F_{\oo}^0(\varphi_{0, t})-\td F^0_{\oo}(s\psi)>-\infty. $$
On the other hand, we have
$$(1-t)(\td I_{\oo_s}(\varphi_{s, t})-\td J_{\oo_s}(\varphi_{s, t}))\leq \int_t^1\;
(\td I_{\oo_s}(\varphi_{s, t})-\td J_{\oo_s}(\varphi_{s, t}))dt\leq -\lim_{t\ri 1^-}\td F^0_{\oo_s}(\varphi_{s, t})$$
and thus we have
$$\lim_{t\ri 1^-}(1-t)|\varphi_{s, t}|_{C^0}=0.$$
We can argue as (\ref{lem:theo2:eq5}) to derive
\beq
\lim_{t\ri 1^-}\int_M\; e^{h_s-\varphi_{s, t}}\oo_s^n=V.\label{theo3:eq1}
\eeq
Thus, by the definition of $\td F$ and (\ref{theo3:eq1}) we have
\beq
-\infty<\lim_{t\ri 1^-}\td F_{\oo_s}(\varphi_{s, t})=\lim_{t\ri 1^-}\td F^0_{\oo_s}(\varphi_{s, t})\leq 0,\n
\eeq
and
\beq
\td F_{\oo}(\psi)=\td F_{\oo}(\oo_1)\geq\lim_{t\ri 1^-}\td F_{\oo}(\oo_{1, t})=\lim_{t\ri 1^-}\td F_{\oo}(\oo_{0, t}),\n
\eeq where we used (\ref{theo3:eq2}) in the last equality. This shows that $\td F$ is bounded from below in the K\"ahler class $2\pi c_1(M).$

Suppose that $\td E_k$ is bounded from below for $\varphi_t$, we can see from the proof of Lemma \ref{lem:theo1} that $\td F$ is also bounded
from below along $\varphi_t$. Thus, $\td F$ is bounded from below in the K\"ahler class $2\pi c_1(M)$ and
by Theorem \ref{theo2}  $\td E_k$ is bounded from below on $\cM_k^+(\oo)$.
The theorem is proved.

\end{proof}

\section{The holomorphic invariant}
Recall that Tian-Zhu defined the holomorphic invariant by
\beq
\cF_X(Y)=\int_M\; Y(h_g-\theta_X(g)) e^{\theta_X(g)}\,\oo_g^n,\quad Y\in \eta(M),\n
\eeq
which generalized the Futaki invariant. Here $\eta(M)$ denotes the space of holomorphic
vector fields on $M.$ Let $\{\Phi(t)\}_{|t|<\infty}$ be the one-parameter
subgroup of automorphisms induced by $Re(Y),$ and $\varphi(x, t)$ be the K\"ahler potential satisfying
\beq
\oo_{\varphi}=\Phi^*_t\oo=\oo+\pbp\varphi. \label{eq:a1}
\eeq
Differentiating (\ref{eq:a1}), we have
\beq
L_{Re(Y)}\oo_{\varphi}=\pbp \pd {\varphi}t.
\eeq
On the other hand, we have $L_Y \oo_{\varphi}=\pbp \theta_Y(\varphi)$. Thus,
\beq
\pd {\varphi}t=Re(\theta_Y(\varphi))+c
\eeq for some constant $c.$ Recall that $u$ satisfies
$$\pbp u=-Ric(\oo_{\varphi})+\oo_{\varphi}+\pbp \theta_X(\varphi).$$
Taking the interior product on both sides, we have
$$Y(u)=\Delta \theta_Y(\varphi)+\theta_Y(\varphi)+Y\theta_X(\varphi).$$
By the definition of $u,$ we have
\beq \pd ut=\Delta\pd {\varphi}t+\pd {\varphi}t+X(\pd {\varphi}t)+c=Re(Yu)+c,\label{eq:a2}\eeq
where we used the fact that $Y\theta_X(\varphi)=X\theta_Y(\varphi).$ Following a direct calculation, we have
the lemma:

\begin{lem}\label{lem1}
Let $\{\Phi(t)\}_{|t|<\infty}$ be the one-parameter
subgroup of automorphisms induced by $Re(Y),$ we have
\beq
\frac d{dt}\td E_0(\varphi)=\frac nV\cF_X(Y),\label{eq:a3}
\eeq
where $\varphi$ is given by $\Phi^*_t\oo=\oo+\pbp\varphi.$
\end{lem}

The main result in this section is

\begin{theo} Let $\{\Phi(t)\}_{|t|<\infty}$ be the one-parameter
subgroup of automorphisms induced by $Re(Y),$ we have
$$\frac d{dt}\td E_k(\varphi)=\frac {(k+1)n}V\cF_X(Y),$$
where $\varphi$ is given by $\Phi^*_t\oo=\oo+\pbp\varphi.$
\end{theo}

\begin{proof}By the definition of $\td E_k$ and Lemma \ref{lem1}, it suffices to check that for any
$k=0, \cdots, n$
\beq
\frac {d}{dt}\td G_k(\varphi)=0.\label{eq:a4}
\eeq
Direct calculation shows that
\beqn
\frac {d}{dt}\td G_k(\varphi)&=&-\frac 1V\frac {d}{dt}\int_M\; \i\p u\wedge \bar\p u\wedge
(\pbp
u)^{k-1}\wedge e^{\theta_X(\varphi)}\oo_{\varphi}^{n-k}\n\\
&=&-\frac 1VRe\int_M\; 2\i\p\, Yu\wedge \bar\p u\wedge(\pbp
u)^{k-1}\wedge e^{\theta_X(\varphi)}\oo_{\varphi}^{n-k}\n\\
&&-\frac 1VRe\int_M\;(k-1)\i\p u\wedge \bar\p u\wedge
(\pbp
u)^{k-2}\wedge \pbp Yu\wedge e^{\theta_X(\varphi)}\oo_{\varphi}^{n-k}\n\\
&&-\frac 1VRe\int_M\;X\theta_Y(\varphi)\i\p u\wedge \bar\p u\wedge
(\pbp
u)^{k-1}\wedge e^{\theta_X(\varphi)}\oo_{\varphi}^{n-k}\n\\
&&-\frac 1VRe\int_M\;(n-k)\i\p u\wedge \bar\p u\wedge
(\pbp
u)^{k-1}\wedge e^{\theta_X(\varphi)}\oo_{\varphi}^{n-k-1}\wedge \pbp \theta_Y(\varphi)\n\\
&=&Re(I_1+I_2+I_3+I_4),\n
\eeqn where $I_i(1\leq i\leq 4)$ denote the integrations on the right hand side respectively.
On the other hand, we have
\beqn 0
&=&\frac 1V\int_M\; i_Y\Big(\p u\wedge(\pbp u)^k\wedge e^{\theta_X(\varphi)}\oo_{\varphi}^{n-k} \Big)\n\\
&=&\frac 1V\int_M\; Yu(\pbp u)^k\wedge e^{\theta_X(\varphi)}\oo_{\varphi}^{n-k} \n\\
&&-\frac 1V\int_M\; k\i \p u\wedge \bar \p Yu\wedge(\pbp u)^{k-1}\wedge e^{\theta_X(\varphi)}\oo_{\varphi}^{n-k}\n\\
&&-\frac 1V\int_M\;(n-k)\i \p u\wedge\bar \p \theta_Y(\varphi)\wedge(\pbp u)^k\wedge e^{\theta_X(\varphi)}\oo_{\varphi}^{n-k-1}\n\\
&=&J_1+J_2+J_3.\n
\eeqn
Note that
\beqs J_1&=&\frac 1V\int_M\; -\i\p Yu\wedge \bar\p u \wedge(\pbp u)^{k-1}\wedge e^{\theta_X(\varphi)}\oo_{\varphi}^{n-k}\\
&&-\frac 1V\int_M\;Yu\i\p\theta_X(\varphi)\wedge \bar \p u\wedge(\pbp u)^{k-1}\wedge e^{\theta_X(\varphi)}\oo_{\varphi}^{n-k}\\
&=&J_{1a}+J_{1b}
\eeqs
and
\beqs
J_3&=&\frac 1V\int_M\; -(n-k)\p u\wedge \bar \p u\wedge \pbp \theta_Y(\varphi)\wedge(\pbp u)^{k-1}\wedge e^{\theta_X(\varphi)}\oo_{\varphi}^{n-k-1}\\
&&+\frac 1V\int_M\; -(n-k)\p \theta_X(\varphi)\wedge \bar \p \theta_Y(\varphi)\wedge \bar\p u\wedge \bar \p u\wedge(\pbp u)^{k-1}\wedge e^{\theta_X(\varphi)}\oo_{\varphi}^{n-k-1}\\
&=&J_{3a}+J_{3b}.
\eeqs

Now we calculate
\beqs
0&=&\frac 1V\int_M\; -i_Y\Big(\p \theta_X(\varphi)\wedge\i\p u\wedge \bar \p u\wedge(\pbp u)^{k-1}\wedge e^{\theta_X(\varphi)}\oo_{\varphi}^{n-k}\Big)\n\\
&=&\frac 1V\int_M\; -Y\theta_X(\varphi)\i\p u\wedge \bar \p u\wedge(\pbp u)^{k-1}\wedge e^{\theta_X(\varphi)}\oo_{\varphi}^{n-k}\n\\
&&+\frac 1V\int_M\; Yu\, \i\p \theta_X(\varphi)\wedge \bar \p u\wedge(\pbp u)^{k-1}\wedge e^{\theta_X(\varphi)}\oo_{\varphi}^{n-k}\\
&&+\frac 1V\int_M\; (k-1)\p \theta_X(\varphi)\wedge \i \p u\wedge \bar \p u\wedge(\pbp u)^{k-2}\wedge \i \bar \p Yu \wedge e^{\theta_X(\varphi)}\oo_{\varphi}^{n-k}\n\\
&&+\frac 1V\int_M\; (n-k)\p \theta_X(\varphi)\wedge\i\p u\wedge \bar \p u\wedge(\pbp u)^{k-1}\wedge e^{\theta_X(\varphi)}\oo_{\varphi}^{n-k-1}\wedge \i\bar \p \theta_Y(\varphi)\\
&=&K_1+K_2+K_3+K_4.
\eeqs
Note that
\beqs
K_3&=&\frac 1V\int_M\; (k-1)\i\p u\wedge \bar \p Yu\wedge (\pbp u)^{k-1}\wedge e^{\theta_X(\varphi)}\oo_{\varphi}^{n-k}\n\\
&&-\frac 1V\int_M\; (k-1)\i\p u\wedge \bar \p u\wedge (\pbp u)^{k-2}\wedge \pbp Yu\wedge e^{\theta_X(\varphi)}\oo_{\varphi}^{n-k}\\
&=&K_{3a}+K_{3b}.
\eeqs
Combining these equalities, we have the following relations:
$$I_1=2J_{1a},\quad I_2=K_{3b},\quad I_3=K_1, \quad I_4=J_{3a}$$
and
$$J_{1b}+K_2=0,\quad J_{3b}+K_4=0,\quad J_2+K_{3a}=J_{1a}.$$
Thus, we have
$$\sum_{i=1}^4I_i=\sum_{i=1}^3J_i+\sum_{i=1}^4K_i=0.$$
The theorem is proved.
\end{proof}

\noindent
Department of Mathematics, \\
East China Normal University, Shanghai, 200241, China.\\
Email: lihaozhao@gmail.com


\begin{thebibliography}{2}
\bibitem{[BaMa]}S. Bando, T. Mabuchi. Uniqueness of Einstein K\"ahler metrics modulo connected group actions.  Algebraic geometry, Sendai, 1985,  11--40, Adv. Stud. Pure Math., 10, North-Holland, Amsterdam, 1987.
\bibitem{[CTZ]}H. D. Cao, G. Tian, X. H. Zhu. K\"ahler Ricci
solitions on compact complex manifolds with $c_1(M)>0$,  Geom.
Funct. Anal., 15(2005),  697-719.

\bibitem{[chen1]}X. X. Chen.  On the lower bound of energy functional $E_1 (I)$--
 a stability theorem on the K\"ahler Ricci flow. J. Geometric
 Analysis. 16 (2006) 23-38.

 \bibitem{[Chen4]}X. X. Chen, Space of K\"ahler metrics (IV)--On the lower bound of the $K$-energy. arXiv:0809.4081.

\bibitem{[CLW]}X. X. Chen, H. Li, B. Wang. Kahler-Ricci flow with small initial  energy, Geom. Func. Anal., Vol18, No 5(2009),
1525-1563.


\bibitem{[chen-tian1]}X. X. Chen, G. Tian. Ricci flow on K\"ahler-Einstein surfaces.  Invent. Math.  147  (2002),  no. 3, 487--544.
\bibitem{[chen-tian2]}X. X. Chen, G. Tian. Ricci flow on K\"ahler-Einstein
manifolds. Duke. Math. J.  131, (2006),  no. 1, 17-73.


\bibitem{[DT]}W. Y. Ding and G. Tian, The generalized Moser-Trudinger inequality, in Nonlinear
Analysis and Microlocal Analysis: Proceedings of the International
Conference at Nankai Institute of Mathematics (K.-C. Chang et al.,
Eds.), World Scientific, 1992, 57-70. ISBN 9810209134.



\bibitem{[L]}H. Z. Li,  A new formula for the Chen-Tian energy functionals $E_k$ and its applications, International Mathematics Research Notices, Vol. 2007, Article ID rnm033, 17 pages, 2007 £¨SCI£©


\bibitem{[L2]}H. Z. Li, On the lower bound of F functional and K
energy, Osaka J. Math. Volume 45, Number 1 (2008), 253-264.

\bibitem{[Liu]}C. J. Liu. Bando-Futaki Invariants on Hypersurfaces. math.DG/0406029.

\bibitem{[Ma]}T. Mabuchi. $K$-energy maps integrating Futaki invariants. Tohoku Math. J. (2) 38(1986), no. 4, 575-593.
\bibitem{[Pali]}N. Pali. A consequence of a lower bound of the $K$-energy.  Int. Math. Res. Not.  2005,  no. 50, 3081--3090.


\bibitem{[R]}Y. Rubinstein, On energy functionals, K\"ahler-Einstein metrics, and the Moser-Trudinger-Onofri
neighborhood, J. Func. Anal. 255(2008), no. 1, 641-2660.

\bibitem{[SoWe]}J. Song, B. Weinkove. Energy functionals and canonical Kahler metrics. Duke Math. J.
137(2007), no. 1, 159-184.
\bibitem{[Tian1]}G. Tian. On K\"ahler-Einstein metrics on certain K\"ahler manifolds with $C\sb 1(M)>0$.  Invent. Math.  89  (1987),  no. 2, 225--246.
\bibitem{[Tian2]}G. Tian. K\"ahler-Einstein metrics with positive scalar curvature.  Invent. Math.  130  (1997),  no. 1, 1--37.

\bibitem{[Tian3]}G. Tian. Canonical metrics in K\"ahler geometry. Notes taken by Meike Akveld. Lectures in Mathematics ETH Z\"urich. Birkh\"auser Verlag, Basel, 2000.

\bibitem{[TZ]}G. Tian, X. H. Zhu, A new holomorphic invariant
and uniqueness of K\"ahler-Ricci solitons. Comment. Math. Helv.
77(2002), 297-325.
\bibitem{[TZ2]}G. Tian, X. H. Zhu, Uniqueness of K\"ahler-Ricci solitons, Acta Math. 184
(2000), 271-305.
\bibitem{[Z]}X. H. Zhu, K\"ahler-Ricci soliton typed equations on compact complex manifolds
with $C_1(M)>0$, J. Geometric Analysis 10 (2000), 747-762.







\end{thebibliography}
\end{document}